\documentclass[12pt,reqno]{amsart}
\usepackage{amssymb,amsmath,amsthm,hyperref}
\newcommand{\sg}{\textnormal{sg}}

\newtheorem{theorem}{Theorem}
\newtheorem{lemma}[theorem]{Lemma}
\newtheorem{corollary}[theorem]{Corollary}

\theoremstyle{remark}

\theoremstyle{definition}

\newtheorem{definition}[theorem]{Definition}

\numberwithin{theorem}{section} \numberwithin{equation}{section}
\numberwithin{example}{section}

\title{Ramanujan's $_{1}\psi_1$ summation, Hecke-type double sums, and Appell-Lerch sums}

\author{Eric Mortenson}

\begin{document}

\date{29 December 2011}

\subjclass[2000]{11B65, 11F11, 11F27}

\keywords{Hecke-type double sums, Appell-Lerch sums, mock theta functions, indefinite theta series}

\begin{abstract}
We use a specialization of Ramanujan's ${}_1\psi_1$ summation to give a new proof of a recent formula of Hickerson and Mortenson which expands a special family of Hecke-type double sums in terms of Appell-Lerch sums and theta functions.
\end{abstract}

\address{Department of Mathematics, The University of Queensland,
St Lucia, QLD 4072, Australia}
\email{mort@maths.uq.edu.au}
\maketitle
\setcounter{section}{-1}

\section{Notation}\label{section:notation}

 Let $q$ be a nonzero complex number with $|q|<1$ and define $\mathbb{C}^*:=\mathbb{C}-\{0\}$.  We recall notation for the theta function
\begin{gather*}
 j(x;q):=(x)_{\infty}(q/x)_{\infty}(q)_{\infty}=\sum_{n}(-1)^nq^{\binom{n}{2}}x^n,\\
 {\text{and }}\ \ j(x_1,x_2,\dots,x_n;q):=j(x_1;q)j(x_2;q)\cdots j(x_n;q).
\end{gather*}
where in the first line the equivalence of product and sum follows from Jacobi's triple product identity.  We will frequently use special cases of the above definition.  Here $a$ and $m$ are integers with $m$ positive.  We define
\begin{gather*}
J_{a,m}:=j(q^a;q^m), \ \ \overline{J}_{a,m}:=j(-q^a;q^m), \ {\text{and }}J_m:=J_{m,3m}=\prod_{n\ge 1}(1-q^{mn}).
\end{gather*}
We will also use standard basic hypergeometric series notation for finite and infinite products, see \cite{GR}.

\section{Introduction}\label{section:intro}

In his last letter to Hardy, Ramanujan gave a list of seventeen functions which he called ``mock $\vartheta$-functions.'' Ramanujan's list was divided into four groups of functions which he described as being of orders $3$, $5$, $5$, and $7$, with each function defined as a $q$-series convergent for $|q|<1.$  Specifically, he defined the functions in terms of Eulerian forms, i.e. $q$-hypergeometric series.  For example, the fifth order mock $\vartheta$-function $f_0(q)$ reads
\begin{align*}
f_0(q)&:=\sum_{n= 0}^{\infty}\frac{q^{n^2}}{(-q)_n}=1+\frac{q}{(1+q)}+\frac{q^4}{(1+q)(1+q^2)}+\cdots.
\end{align*}
Ramanujan also stated that the mock $\vartheta$-functions have asymptotic properties as $q$ approaches a root of unity, similar to those of ordinary theta functions, but that they are not theta functions.

Historically, problems for the mock $\vartheta$-functions have involved determining asymptotic properties of the Fourier coefficients, proving identities between the mock $\vartheta$-functions, and determining the modularity properties.  Eulerian forms are difficult to work with, so progress on the classical problems was not made until after techniques had been introduced which converted the Eulerian forms into alternate forms such as Appell-Lerch sums, Hecke-type double sums, and Fourier coefficients of meromorphic Jacobi forms.  However, the techniques employed were not robust.

To prove identities between the third order functions and to find their modularity properties, Watson used a formula from basic hypergeometric series to convert the Eulerian forms to what were essentially Appell-Lerch sums \cite{W3}.  Watson's techniques did not work for fifth and seventh order functions, so Andrews used a (then) little-known lemma of Bailey to convert the Eulerian forms of the fifth and seventh orders into Hecke-type double sums \cite{A}.  Hickerson then used the constant term method to convert from the Hecke-type forms to what were again essentially Appell-Lerch sums.  As a result, Hickerson was able to prove the mock theta conjectures for the fifth orders \cite{H5} and the analogous identities for the seventh orders \cite{H7}.  We recall that the mock theta conjectures and the analogous identities for the seventh order functions are identities which express Eulerian forms in terms of the universal mock $\vartheta$-function
\begin{equation*}
g(x,q):=x^{-1}\Big ( -1 +\sum_{n=0}^{\infty}\frac{q^{n^2}}{(x)_{n+1}(q/x)_{n}} \Big ).
\end{equation*}
For example, the mock theta conjecture for the fifth order $f_0(q)$ reads 
\begin{align*}
f_0(q)&:=\sum_{n= 0}^{\infty}\frac{q^{n^2}}{(-q)_n}=\frac{J_{5,10}J_{2,5}}{J_1}-2q^2g(q^2,q^{10}).
\end{align*}
In a breakthrough result \cite{Zw}, Zwegers showed that Appell-Lerch sums, Hecke-type double sums, and Fourier coefficients of meromorphic Jacobi forms all exhibit the same near-modular behaviour.  Among other applications, Zwegers' result gives a robust method for proving identities between mock $\vartheta$-functions \cite{Za,F}.

In joint work with Dean Hickerson, the author found and proved a master formula which expands a family of Hecke-type double sums in terms of Appell-Lerch sums and theta functions \cite[Theorem $1.6$]{HM}.  In order to state the expansion formula of the special family of Hecke-type double sums, we need to define some terms.   We will use the following definition of an Appell-Lerch sum:
\begin{definition}  \label{definition:mdef} Let $x,z\in\mathbb{C}^*$ with neither $z$ nor $xz$ an integral power of $q$. Then
\begin{equation}
m(x,q,z):=\frac{1}{j(z;q)}\sum_{r=-\infty}^{\infty}\frac{(-1)^rq^{\binom{r}{2}}z^r}{1-q^{r-1}xz}.\label{equation:mdef-eq}
\end{equation}
\end{definition}

\noindent Our definition for the building block of Hecke-type double sums reads 
\begin{definition} \label{definition:fabc-def}  Let $x,y\in\mathbb{C}^*$. Then
\begin{equation}
f_{a,b,c}(x,y,q):=\sum_{\substack{\sg(r)=\sg(s)}} \sg(r)(-x)^r(-y)^sq^{a\binom{r}{2}+brs+c\binom{s}{2}},\label{definition:f-def}\\
\end{equation}
where $\sg(r):=1$ if $r\ge0$, and $\sg(r)=-1$ if $r<0$.
\end{definition}
\noindent We also define the following expression involving Appell-Lerch sums:
\begin{align}
g_{a,b,c}&(x,y,q)\notag\\
:=&\sum_{t=0}^{a-1}(-y)^tq^{c\binom{t}{2}}j(q^{bt}x;q^a)m\Big (-q^{a\binom{b+1}{2}-c\binom{a+1}{2}-t(b^2-ac)}\frac{(-y)^a}{(-x)^b},q^{a(b^2-ac)},-1\Big )\label{equation:mdef-2}\\
&+\sum_{t=0}^{c-1}(-x)^tq^{a\binom{t}{2}}j(q^{bt}y;q^c)m\Big (-q^{c\binom{b+1}{2}-a\binom{c+1}{2}-t(b^2-ac)}\frac{(-x)^c}{(-y)^b},q^{c(b^2-ac)},-1\Big ).\notag
\end{align}
We can now state the formula of \cite{HM} which expands a special family of Hecke-type double sums in terms of Appell-Lerch sums and theta functions.
\begin{theorem}[\cite{HM}, Theorem $1.6$]   \label{theorem:masterFnp} Let $n$ and $p$ be positive integers with $(n,p)=1$.  For generic $x,y\in \mathbb{C}^*$
\begin{align*}
f_{n,n+p,n}(x,y,q)=g_{n,n+p,n}(x,y,q)+\theta_{n,p}(x,y,q),
\end{align*}
where
\begin{align*}
&\theta_{n,p}(x,y,q):=\sum_{r^*=0}^{p-1}\sum_{s^*=0}^{p-1}q^{n\binom{r-(n-1)/2}{2}+(n+p)\big (r-(n-1)/2\big )\big (s+(n+1)/2\big )+n\binom{s+(n+1)/2}{2}} (-x)^{r-(n-1)/2}\\
& 
 \cdot \frac{(-y)^{s+(n+1)/2}J_{p^2(2n+p)}^3j(-q^{np(s-r)}x^n/y^n;q^{np^2})j(q^{p(2n+p)(r+s)+p(n+p)}x^py^p;q^{p^2(2n+p)})}{\overline{J}_{0,np(2n+p)}j(q^{p(2n+p)r+p(n+p)/2}(-y)^{n+p}/(-x)^n,q^{p(2n+p)s+p(n+p)/2}(-x)^{n+p}/(-y)^n;q^{p^2(2n+p)})}.\notag
\end{align*}
Here $r:=r^*+\{(n-1)/2\}$ and $s:=s^*+\{ (n-1)/2\}$, with $0\le \{ \alpha\}<1$ denoting the fractional part of $\alpha$. 
\end{theorem}

 Not only does Theorem \ref{theorem:masterFnp} give a robust method for proving identities between mock $\vartheta$-functions, but it also gives a robust method for converting Hecke-type double sums into Appell-Lerch sums.  Moreover, the formula also finds mock theta conjecture-like identities.  In recent work \cite{A2}, Andrews expands special families of Eulerian forms in terms of $q$-orthogonal polynomials.  His expansion formulas simultaneously prove identities of Rogers-Ramanujan type as well as Hecke-type double sum expansions for mock $\vartheta$-functions.  For example, Andrews' formula \cite[$(1.19)$]{A2} proves Slater's identity \cite[$(39)$]{S}
 \begin{align*}
\sum_{n=0}^{\infty}\frac{q^{2n^2}}{(q;q)_{2n}}=\frac{\overline{J}_{3,8}}{J_2},
\end{align*}
and also proves a Hecke-type double sum expansion for a new mock $\vartheta$-function \cite[$(1.14)$]{A2}
\begin{align*}
\sum_{n=0}^{\infty}\frac{q^{2n^2}}{(-q;q)_{2n}}&=\frac{1}{(q^2;q^2)_{\infty}}\sum_{n=0}^{\infty}q^{4n^2+n}(1-q^{6n+3})\sum_{j=-n}^{n}(-1)^jq^{-j^2}.
\end{align*}
Using Theorem \ref{theorem:masterFnp}, the author showed the following new mock theta conjecture-like identity \cite{M}:
\begin{align*}
\sum_{n=0}^{\infty}\frac{q^{2n^2}}{(-q;q)_{2n}}=2-2qg(-q,q^{8})-\frac{J_{1,2}\overline{J}_{3,8}}{J_2}.
\end{align*}
In \cite{M}, the author also used Theorem \ref{theorem:masterFnp} to find new mock theta conjecture-like identities for another new mock $\vartheta$-function of Andrews \cite{A2} as well as for two new mock $\vartheta$-functions of Bringmann, et al, \cite{BHL}.  

Theorem \ref{theorem:masterFnp} can also be used to find more exotic identities.  In other recent work of Andrews \cite{A3}, mock $\vartheta$-functions are obtained from identities of Rogers-Ramanujan type by interchanging the role of evens and odds in the partition theoretic interpretations of the Eulerian forms.  Andrews thus found the following Hecke-type double sum expansion \cite[$(4.25)$]{A3}, which is related to mock theta functions:
\begin{align*}
\sum_{n= 0}^{\infty}\frac{q^{3n^2+2n}}{(q)_{2n}(-q^2;q^2)_n}
&=\frac{1}{(q^2;q^2)_{\infty}}\sum_{n=0}^{\infty}q^{4n^2+2n}(1-q^{4n+2})\sum_{j=-n}^n(-1)^j(-q)^{-j(3j-1)/2}. 
\end{align*}
Using Theorem \ref{theorem:masterFnp} as a guide, we are quickly led to
\begin{equation}
\sum_{n=0}^{\infty}\frac{q^{3n^2+2n}}{(q)_{2n}(-q^2;q^2)_n}
=-q^2g(q^2,q^{10})\cdot \frac{j(-q;-q^5)}{J_2}+q^3g(q^4,q^{10})\cdot \frac{j(q^2;-q^5)}{J_2}+\frac{j(-q^5;-q^{15})^3}{J_2J_{10}},
\end{equation}
which holds numerically.

Theorem \ref{theorem:masterFnp} was originally shown in \cite{HM} by computing functional equations and then checking poles and residues.  In this note we give a new proof of Theorem \ref{theorem:masterFnp}, that uses a special case of Ramanujan's ${}_1\psi_1$ summation.  The new proof sheds light on the structure of the $\theta_{n,p}$ expression and demonstrates an amount of control over the number of terms involved.  In Section \ref{section:prelim} we recall the special case of Ramanujan's $_{1}\psi_1$ summation and prove some technical results needed for the new proof of Theorem \ref{theorem:masterFnp}, which we carry out in Section \ref{section:proofofthm}.


\section{Preliminaries}\label{section:prelim}

 We first recall Ramanujan's ${}_1\psi_1$ summation formula.
\begin{theorem} If $|{b}/{a}|<|x|<1$, then 
\begin{equation}
\sum_{n=-\infty}^{\infty}\frac{(a)_n}{(b)_n}x^n=\frac{(b/a,q/ax,ax,q)_{\infty}}{(b,b/ax,q/a,x)_{\infty}}. \label{equation:1psi1}
\end{equation}
\end{theorem}
\noindent We will use the following corollary which follows by setting $a=y$ and $b=qy$.  See, for example \cite[Theorem $1.5$]{H5}.
\begin{corollary} \label{theorem:H5Thm1.5}For $|q|<|x|<1$ and $|q|<|y|<1$,
\begin{equation}
\sum_{r,s}\textup{sg}(r,s)q^{rs}x^ry^s=\frac{J_1^3j(xy;q)}{j(x;q)j(y;q)},\label{equation:ThmH51.5}
\end{equation}
where $\sg(r,s):=\Big ( \sg(r)+\sg(s)\Big )/2.$

\end{corollary}

We now derive some technical results, which will be used in the proof of Theorem \ref{theorem:masterFnp}.  The technical results will allow us to expand Appell-Lerch sums and quotients of theta functions in terms of multiple sums of $q$-series, i.e. sums like the left-hand side of (\ref{equation:ThmH51.5}).  The following two lemmas are for the $n$ odd case of the proof of Theorem \ref{theorem:masterFnp}.  The $n$ even versions are similar and will be omitted.

\begin{lemma} \label{lemma:use1psi1} Let $x,y\in \mathbb{C}^*$, $n$ and $p$ positive integers with $n$ odd, and $r$ an integer with $0\le r \le p-1$.  If
\begin{equation*}
|q^{p(3n+p)/2}|<|x^{-n}y^{n+p}|<|q^{-p(n+p)/2}|
\end{equation*}
then
\begin{equation*}
|q^{p^2(2n+p)}|<|q^{p(n+p)/2+p(2n+p)r}x^{-n}y^{n+p}|<1.
\end{equation*}
\end{lemma}
\begin{proof}[Proof of Lemma \ref{lemma:use1psi1}]  We have that
\begin{equation}
|x^{-n}y^{n+p}|<|q^{-p(n+p)/2}|\cdot |q^{-p(2n+p)r}| \label{equation:use1psi1-a}
\end{equation}
and
\begin{equation}
|x^{-n}y^{n+p}|>|q^{p(3n+p)/2}|\cdot |q^{p(2n+p)(p-1-r)}|.\label{equation:use1psi1-b}
\end{equation}
Combining (\ref{equation:use1psi1-a}) and (\ref{equation:use1psi1-b}) gives the result.
\end{proof}

\begin{lemma} \label{lemma:mxqz-expansion-conditions}Let $x,y\in \mathbb{C}^*$, $n$ and $p$ positive integers with $n$ odd, and $k$ an integer with $0\le k \le n-1$.  If
\begin{equation*}
|q^{p(3n+p)/2}|<|x^{-n}y^{n+p}|<|q^{-p(n+p)/2}|
\end{equation*}
then for $0\le k \le (n-1)/2$ we have
\begin{equation*}
|q^{np(2n+p)}|<|q^{n(np+\binom{p+1}{2})-kp(2n+p)}x^{n}y^{-n-p}|<1,
\end{equation*}
and for  $(n+1)/2 \le k \le n-1$ we have
\begin{equation*}
|q^{np(2n+p)}|<|q^{n(np+\binom{p+1}{2})+np(2n+p)-kp(2n+p)}x^{n}y^{-n-p}|<1.
\end{equation*}
\end{lemma}
\begin{proof}[Proof of Lemma \ref{lemma:mxqz-expansion-conditions}]
Arguing as in the proof of Lemma \ref{lemma:use1psi1}, we obtain
\begin{equation*}
|q^{np(2n+p)}|<|q^{p(3n+p)/2+tp(2n+p)}x^ny^{-n-p}|<1,
\end{equation*}
where $t$ is an integer with $0\le t \le n-1.$  The lemma then follows by setting
\begin{equation*}
t=(n-1)/2 -k
\end{equation*}
when $0\le k \le (n-1)/2$ and by setting
\begin{equation*}
t=n+(n-1)/2 -k
\end{equation*}
when $(n+1)/2\le k \le n-1$.
\end{proof}

\begin{lemma} \label{lemma:mxqz-expansionA}If $|q|<|x|<1$, then
\begin{align}
&\overline{J}_{0,1}m(x,q,-1)=\sum_{v,s}\textup{sg}(v,s)q^{\binom{v+1}{2}+vs}(-x)^s.\label{equation:exp1}
\end{align}
\end{lemma}

\begin{proof}[Proof of Lemma \ref{lemma:mxqz-expansionA}] By Definition \ref{definition:mdef}, we have
\begin{align*}
\overline{J}_{0,1}m(x,q,-1)&=\sum_{k}\frac{(-1)^kq^{\binom{k}{2}}(-1)^k}{1-q^{k-1}(-x)}
=\sum_{k}\frac{q^{\binom{k+1}{2}}}{1-q^{k}(-x)}.
\end{align*}
Noting that $|q^kx|<1$ if and only if $k\ge 0$, the result then follows from the geometric series.
\end{proof}
\begin{lemma} \label{lemma:mxqz-expansionB}If $|q|<|qx|<1$, then
\begin{align}
&\overline{J}_{0,1}m(x,q,-1)=\sum_{v,s}\textup{sg}(v,s)q^{\binom{v+1}{2}+(v+1)(s+1)}(-x)^s.\label{equation:exp2}
\end{align}
\end{lemma}
\begin{proof}[Proof of Lemma \ref{lemma:mxqz-expansionB}]  This is similar to the proof of Lemma  \ref{lemma:mxqz-expansionA}.
\end{proof}

\begin{lemma} \label{lemma:sign-ids} Let $n$ be a positive integer and $r$, $s$, $k$, $w$ integers.  Then 
\begin{equation*}
\textup{sg}(nr+k+nw+\lfloor n/2\rfloor )=
\begin{cases}
-\textup{sg}(-w-1-r) & {\text{ if } 0\le k \le \lfloor n/2\rfloor},\\
-\textup{sg}(-w-2-r) & {\text{ if } \lfloor n/2\rfloor +1\le k \le n-1},
\end{cases}
\end{equation*}
and
\begin{equation*}
\textup{sg}(ns+k-nw-\lfloor n/2\rfloor -1)=
\begin{cases}
-\textup{sg}(w-s) & {\text{ if } 0\le k \le \lfloor n/2\rfloor },\\
-\textup{sg}(w-1-s) & {\text{ if } \lfloor n/2\rfloor +1\le k \le n-1.}
\end{cases}
\end{equation*}
\end{lemma}
\begin{proof}[Proof of Lemma \ref{lemma:sign-ids}] This is a straightforward result of the fact that $\sg(-1-r)=-\sg(r).$
\end{proof}


\section{proof of theorem \ref{theorem:masterFnp}}\label{section:proofofthm}

We prove the case $n$ odd; the case $n$ even is similar and will be omitted.  We rewrite the statement of Theorem \ref{theorem:masterFnp}:
\begin{align}
&\overline{J}_{0,np(2n+p)}\Big ( f_{n,n+p,n}(x,y,q)-g_{n,n+p,n}(x,y,q)\Big )\label{equation:nodd} \\
&=\sum_{r=0}^{p-1}\sum_{s=0}^{p-1}q^{n\binom{r-(n-1)/2}{2}+(n+p)\big (r-(n-1)/2\big )\big (s+(n+1)/2\big )+n\binom{s+(n+1)/2}{2}} (-x)^{r-(n-1)/2}(-y)^{s+(n+1)/2}\notag\\
& \ \ \
 \cdot \frac{J_{p^2(2n+p)}^3j(-q^{np(s-r)}x^n/y^n;q^{np^2})j(q^{p(2n+p)(r+s)+p(n+p)}x^py^p;q^{p^2(2n+p)})}{j(q^{p(2n+p)r+p(n+p)/2}(-y)^{n+p}/(-x)^n,q^{p(2n+p)s+p(n+p)/2}(-x)^{n+p}/(-y)^n;q^{p^2(2n+p)})}.\notag
\end{align}
It suffices to prove (\ref{equation:nodd}) in the case

\begin{equation*}
|q^{p(3n+p)/2}|<|x^{-n}y^{n+p}|<|q^{-p(n+p)/2}| \text{ and } |q^{p(3n+p)/2}|<|y^{-n}x^{n+p}|<|q^{-p(n+p)/2}|.
\end{equation*}
The result will then follow from analytic continuation.

 We rewrite the right-hand side of (\ref{equation:nodd}).  By Lemma \ref{lemma:use1psi1} and Corollary \ref{theorem:H5Thm1.5}, the last quotient of the right-hand side of  (\ref{equation:nodd}) equals
\begin{align*}
\sum_{t,u}&\textup{sg}(t,u)q^{p^2(2n+p)tu}\Big ((-1)^pq^{p(2n+p)r+p(n+p)/2}y^{n+p}x^{-n}\Big )^t\Big ((-1)^pq^{p(2n+p)s+p(n+p)/2}x^{n+p}y^{-n}\Big )^u\\
&=\sum_{t,u}\textup{sg}(t,u)(-x)^{(n+p)u-nt}(-y)^{(n+p)t-nu}q^{p^2(2n+p)tu+p(2n+p)(rt+su)+(t+u)p(n+p)/2}.
\end{align*}
We also have that
\begin{equation*}
j(-q^{np(s-r)}x^n/y^n;q^{np^2})=\sum_{v}q^{np^2\binom{v}{2}}(q^{np(s-r)}x^{n}y^{-n})^v
=\sum_{v}q^{np^2\binom{v}{2}+np(s-r)v}x^{nv}y^{-nv},
\end{equation*}
so the right-hand side of (\ref{equation:nodd}) equals
\begin{align}
\sum_{r=0}^{p-1}\sum_{s=0}^{p-1}&\sum_{t,u,v}\textup{sg}(t,u)(-x)^{r+(n+p)u-nt+nv-(n-1)/2}(-y)^{s+(n+p)t-nu-nv+(n+1)/2}\label{equation:rh1}\\
&\cdot q^{n\binom{r-(n-1)/2}{2}+(n+p)\big (r-(n-1)/2\big )\big (s+(n+1)/2\big )+n\binom{s+(n+1)/2}{2}+np^2\binom{v}{2}+np(s-r)v} \notag\\
& \ \ \ \ \ \cdot q^{p^2(2n+p)tu+p(2n+p)(rt+su)+p(t+u)(n+p)/2}.\notag
\end{align}
Now we apply the change of variables
$r=R-pu$, $s=S-pt,$ $v=t-u-w.$  This allows us to remove the sums over $r$ and $s$.  The general term (\ref{equation:rh1}) then becomes
\begin{align}
\sum_{R,S,w}&\textup{sg}(R,S)(-x)^{R-nw-(n-1)/2}(-y)^{S+nw+(n+1)/2}\label{equation:rh2}\\
&\cdot q^{n\binom{R-nw-(n-1)/2}{2}+(n+p)(R-nw-(n-1)/2)(S+nw+(n+1)/2)+n\binom{S+nw+(n+1)/2}{2}+np(2n+p)\binom{w+1}{2}}. \notag
\end{align}
We make one more change of variables:  $R=r+nw+(n-1)/2$, $S=s-nw-(n+1)/2$.  The right-hand side of (\ref{equation:nodd}) then equals
\begin{align}
\sum_{r,s}&(-x)^{r}(-y)^{s} \label{equation:rh3}\\
&\cdot \sum_{w}\textup{sg}(r+nw+(n-1)/2,s-nw-(n+1)/2)q^{n\binom{r}{2}+(n+p)rs+n\binom{s}{2}+np(2n+p)\binom{w+1}{2}}.\notag
\end{align}
We break (\ref{equation:rh3}) up into two parts:
\begin{align}
\tfrac12 \sum_{r,s}&(-x)^{r}(-y)^{s}\sum_{w}\textup{sg}(r+nw+(n-1)/2)q^{n\binom{r}{2}+(n+p)rs+n\binom{s}{2}+np(2n+p)\binom{w+1}{2}} \label{equation:rh4}\\
&+\tfrac12\sum_{r,s}(-x)^{r}(-y)^{s}\sum_{w}\textup{sg}(s-nw-(n+1)/2)q^{n\binom{r}{2}+(n+p)rs+n\binom{s}{2}+np(2n+p)\binom{w+1}{2}}.\notag
\end{align}
In the first line of (\ref{equation:rh4}), we rewrite $r$ as $rn+k$ where $0\le k\le n-1$.  In the second line of (\ref{equation:rh4}), we do the same for $s$.  For the final form of the right-hand side of (\ref{equation:nodd}), we then have
\begin{align}
\tfrac12 \sum_{k=0}^{n-1}&\sum_{r,s}(-x)^{nr+k}(-y)^{s}\label{equation:rh5} \\
& \ \ \ \ \ \ \ \ \cdot \sum_{w}\textup{sg}(nr+k+nw+(n-1)/2)q^{n\binom{nr+k}{2}+(n+p)(nr+k)s+n\binom{s}{2}+np(2n+p)\binom{w+1}{2}} \notag\\
&+\tfrac12\sum_{k=0}^{n-1}\sum_{r,s}(-x)^{r}(-y)^{ns+k} \notag \\ 
& \ \ \ \ \ \ \ \ \cdot \sum_{w}\textup{sg}(ns+k-nw-(n+1)/2)q^{n\binom{r}{2}+(n+p)r(ns+k)+n\binom{ns+k}{2}+np(2n+p)\binom{w+1}{2}}.\notag
\end{align}

We rewrite the left-hand side of (\ref{equation:nodd}).  The first term on the left-hand side of (\ref{equation:nodd}) equals
\begin{align}
\overline{J}_{0,np(2n+p)}&f_{n,n+p,n}(x,y,q)
=\sum_{w}q^{np(2n+p)\binom{w+1}{2}}\cdot \sum_{r,s}\textup{sg}(r,s)(-x)^r(-y)^sq^{n\binom{r}{s}+(n+p)rs+n\binom{s}{2}}\notag\\
&=\sum_{r,s}(-x)^r(-y)^s\sum_{w}\textup{sg}(r,s)q^{n\binom{r}{s}+(n+p)rs+n\binom{s}{2}+np(2n+p)\binom{w+1}{2}}.\label{equation:lh1}
\end{align}
The second term on the left-hand side of (\ref{equation:nodd}) equals
{\allowdisplaybreaks \begin{align}
-&\overline{J}_{0,np(2n+p)}g_{n,n+p,n}(x,y,q)\notag\\
&=- \overline{J}_{0,np(2n+p)}\label{equation:lh2A}\\
& \ \ \ \ \ \ \cdot \Big ( \sum_{k=0}^{n-1}(-x)^kq^{n\binom{k}{2}}j(q^{(n+p)k}y;q^n)m\big (-q^{n(np+\binom{p+1}{2})-kp(2n+p)}\frac{(-x)^n}{(-y)^{n+p}},q^{np(2n+p)},-1\big)\notag\\
& \ \ \ \ \ \ \  + \sum_{k=0}^{n-1}(-y)^kq^{n\binom{k}{2}}j(q^{(n+p)k}x;q^n)m\big (-q^{n(np+\binom{p+1}{2})-kp(2n+p)}\frac{(-y)^n}{(-x)^{n+p}},q^{np(2n+p)},-1\big ) \Big  ).\notag
\end{align}}

\noindent To rewrite (\ref{equation:lh2A}) we use Lemmas \ref{lemma:mxqz-expansion-conditions} and \ref{lemma:mxqz-expansionA} for $0\le k \le (n-1)/2$ as well as  Lemmas \ref{lemma:mxqz-expansion-conditions} and \ref{lemma:mxqz-expansionB} for $(n+1)/2\le k \le n-1.$  Thus (\ref{equation:lh2A}) becomes
\begin{align}
\sum_{k=0}^{(n-1)/{2}}& \sum_{v,s,t}\textup{sg}(v,s)(-x)^{k+ns}(-y)^{t-s(n+p)}q^{n\binom{k}{2}+n\binom{t}{2}+(n+p)kt+n(np+\binom{p+1}{2})s-kp(2n+p)s}\label{equation:lh2B}\\
&   \ \ \ \ \ \cdot q^{np(2n+p)(\binom{v+1}{2}+vs)}\notag\\
&   +\sum_{k=(n+1)/{2}}^{n-1} \ \ \sum_{v,s,t}\textup{sg}(v,s)(-x)^{k+ns}(-y)^{t-s(n+p)}q^{n\binom{k}{2}+n\binom{t}{2}+(n+p)kt+n(np+\binom{p+1}{2})s-kp(2n+p)s}\notag\\
& \ \ \ \ \ \cdot q^{np(2n+p)(\binom{v+1}{2}+(v+1)(s+1))}\notag\\
+&\sum_{k=0}^{(n-1)/{2}} \sum_{v,s,t}\textup{sg}(v,s)(-y)^{k+ns}(-x)^{t-s(n+p)}q^{n\binom{k}{2}+n\binom{t}{2}+(n+p)kt+n(np+\binom{p+1}{2})s-kp(2n+p)s}\notag\\
& \ \ \ \ \ \cdot q^{np(2n+p)(\binom{v+1}{2}+vs)}\notag\\
&  +\sum_{k=(n+1)/{2}}^{n-1}\ \  \sum_{v,s,t}\textup{sg}(v,s)(-y)^{k+ns}(-x)^{t-s(n+p)}q^{n\binom{k}{2}+n\binom{t}{2}+(n+p)kt+n(np+\binom{p+1}{2})s-kp(2n+p)s}\notag\\
& \ \ \ \ \ \cdot q^{np(2n+p)(\binom{v+1}{2}+(v+1)(s+1))}.\notag
\end{align}
Substitute $t=(n+p)s+r$ and rewrite the exponent of $q$ to have
{\allowdisplaybreaks \begin{align}
\sum_{k=0}^{(n-1)/{2}}& \sum_{v,s,r}\textup{sg}(v,s)(-x)^{k+ns}(-y)^{r}q^{n\binom{r}{2}+(n+p)r(ns+k)+n\binom{ns+k}{2}+np(2n+p)\binom{s+v+1}{2}}\label{equation:lh2C}\\
&   +\sum_{k=({n+1})/{2}}^{n-1}\ \  \sum_{v,s,r}\textup{sg}(v,s)(-x)^{k+ns}(-y)^{r}q^{n\binom{r}{2}+(n+p)r(ns+k)+n\binom{ns+k}{2}+np(2n+p)\binom{s+v+2}{2}}\notag\\
+&\sum_{k=0}^{({n-1})/{2}} \sum_{v,s,r}\textup{sg}(v,s)(-y)^{k+ns}(-x)^{r}q^{n\binom{r}{2}+(n+p)r(ns+k)+n\binom{ns+k}{2}+np(2n+p)\binom{s+v+1}{2}}\notag\\
&   +\sum_{k=({n+1})/{2}}^{n-1} \ \ \sum_{v,s,r}\textup{sg}(v,s)(-y)^{k+ns}(-x)^{r}q^{n\binom{r}{2}+(n+p)r(ns+k)+n\binom{ns+k}{2}+np(2n+p)\binom{s+v+2}{2}}.\notag
\end{align}}
Swapping $r$ and $s$ in the first two sums yields
{\allowdisplaybreaks \begin{align}
\sum_{k=0}^{({n-1})/{2}}& \sum_{v,s,r}\textup{sg}(v,r)(-x)^{k+nr}(-y)^{s}q^{n\binom{s}{2}+(n+p)s(nr+k)+n\binom{nr+k}{2}+np(2n+p)\binom{r+v+1}{2}}\label{equation:lh2D}\\
&   +\sum_{k=({n+1})/{2}}^{n-1} \ \ \sum_{v,s,r}\textup{sg}(v,r)(-x)^{k+nr}(-y)^{s}q^{n\binom{s}{2}+(n+p)s(nr+k)+n\binom{nr+k}{2}+np(2n+p)\binom{r+v+2}{2}}\notag\\
+&\sum_{k=0}^{({n-1})/{2}} \sum_{v,s,r}\textup{sg}(v,s)(-y)^{k+ns}(-x)^{r}q^{n\binom{r}{2}+(n+p)r(ns+k)+n\binom{ns+k}{2}+np(2n+p)\binom{s+v+1}{2}}\notag\\
&   +\sum_{k=({n+1})/{2}}^{n-1} \ \ \sum_{v,s,r}\textup{sg}(v,s)(-y)^{k+ns}(-x)^{r}q^{n\binom{r}{2}+(n+p)r(ns+k)+n\binom{ns+k}{2}+np(2n+p)\binom{s+v+2}{2}}.\notag
\end{align}}%
In the four sums, we substitute $v=-w-1-r$, $v=-w-2-r$, $v=w-s$, and $v=w-1-s$ respectively to produce
{\allowdisplaybreaks \begin{align}
&\sum_{k=0}^{({n-1})/{2}} \sum_{w,s,r}\textup{sg}(-w-1-r,r)(-x)^{k+nr}(-y)^{s}q^{n\binom{s}{2}+(n+p)s(nr+k)+n\binom{nr+k}{2}+np(2n+p)\binom{w+1}{2}}\notag\\
&   +\sum_{k=({n+1})/{2}}^{n-1} \ \ \sum_{w,s,r}\textup{sg}(-w-2-r,r)(-x)^{k+nr}(-y)^{s}q^{n\binom{s}{2}+(n+p)s(nr+k)+n\binom{nr+k}{2}+np(2n+p)\binom{w+1}{2}}\notag\\
+&\sum_{k=0}^{({n-1})/{2}} \sum_{w,s,r}\textup{sg}(w-s,s)(-y)^{k+ns}(-x)^{r}q^{n\binom{r}{2}+(n+p)r(ns+k)+n\binom{ns+k}{2}+np(2n+p)\binom{w+1}{2}}\label{equation:lh2E}\\
&   +\sum_{k=({n+1})/{2}}^{n-1} \ \ \sum_{w,s,r}\textup{sg}(w-1-s,s)(-y)^{k+ns}(-x)^{r}q^{n\binom{r}{2}+(n+p)r(ns+k)+n\binom{ns+k}{2}+np(2n+p)\binom{w+1}{2}}.\notag
\end{align}}
 Using Lemma \ref{lemma:sign-ids} then yields (\ref{equation:nodd}).


\section*{Acknowledgements}

We would like to thank Dean Hickerson for helpful comments and suggestions.  We would also like to thank the referee for a careful reading of the manuscript as well as comments and suggestions on the exposition.

\end{document}